\documentclass[a4paper, 12pt, 
]{amsart}
\usepackage{ifpdf}
\ifpdf
\else
  \PassOptionsToPackage{dvipdfmx}{graphicx}
  \PassOptionsToPackage{dvipdfmx}{xcolor}
  \PassOptionsToPackage{dvipdfmx}{hyperref}
\fi

\usepackage{amsthm}

\IfFileExists{ulem.sty}{\usepackage[normalem]{ulem}}{} %
\usepackage{graphicx,xcolor}
\IfFileExists{dsfont.sty}{\usepackage{dsfont}}{}
\usepackage{enumitem}
\usepackage{tikz}
\usepackage{tikz-cd}
\usepackage[utf8]{inputenc}
\usepackage[T1]{fontenc}
\usepackage[all,cmtip]{xy}
\usepackage{todonotes}
\usepackage[titletoc]{appendix}
\usepackage{color}
\usepackage{amsmath,amssymb,amscd,amsfonts}
\usepackage[english]{babel}
\usepackage{amstext}
\usepackage{amsmath}
\usepackage{amsfonts}
\usepackage{latexsym}
\usepackage{ifthen}
\usepackage{xypic}
\IfFileExists{yfonts.sty}{\usepackage{yfonts}}{}
\usepackage[all,cmtip]{xy}
\xyoption{all}
\usepackage{enumitem}

\usepackage[pagebackref]{hyperref}
\hypersetup{
  colorlinks=true,
  linkcolor=magenta,
  urlcolor=red,
  citecolor=blue
}

\setlist[enumerate]{label=(\thethm.\arabic*), before={\setcounter{enumi}{\value{equation}}}, after={\setcounter{equation}{\value{enumi}}}}

\newcommand{\Supp}[0]{\operatorname{Supp}}

\newcommand{\Id}{\mathrm{Id}}

\newcommand{\deldel}{\sqrt{-1}\partial \overline{\partial}}
\newcommand{\e}{\varepsilon}

\newcommand{\idelbar}{\sqrt{-1}\,\partial\bar\partial}

\newcommand{\R}{\mathbb{R}}

\renewcommand{\epsilon}{\varepsilon}

\newcommand{\la}{\langle}

\newcommand{\ra}{\rangle}

\renewcommand{\leq}{\leqslant}
\renewcommand{\geq}{\geqslant}

\newcommand{\dbar}{\bar \partial}

\DeclareMathOperator{\id}{id} %

\newtheorem{thm}{Theorem}[section]

\theoremstyle{remark}

\theoremstyle{definition}

\newtheorem{defn}[thm]{Definition}

\theoremstyle{definition} % ←本文がnormal font（ローマン体）になる

\numberwithin{equation}{thm}

\theoremstyle{remark}
\newtheorem{remark}[thm]{Remark}

\IfFileExists{MnSymbolC10.tfm}{%
\DeclareFontFamily{U}{MnSymbolC}{}
\DeclareSymbolFont{MnSyC}{U}{MnSymbolC}{m}{n}
\DeclareFontShape{U}{MnSymbolC}{m}{n}{
<-6>  MnSymbolC5
<6-7>  MnSymbolC6
<7-8>  MnSymbolC7
<8-9>  MnSymbolC8
<9-10> MnSymbolC9
<10-12> MnSymbolC10
<12->   MnSymbolC12}{}
\DeclareMathSymbol{\intprod}{\mathbin}{MnSyC}{'270}
}{%
\newcommand{\intprod}{\mathbin{\lrcorner}}
}

\makeatletter
\newcommand*{\da@rightarrow}{\mathchar"0\hexnumber@\symAMSa 4B }
\newcommand*{\da@leftarrow}{\mathchar"0\hexnumber@\symAMSa 4C }
\newcommand*{\xdashrightarrow}[2][]{%
  \mathrel{%
    \mathpalette{\da@xarrow{#1}{#2}{}\da@rightarrow{\,}{}}{}%
  }%
}
\newcommand{\xdashleftarrow}[2][]{%
  \mathrel{%
    \mathpalette{\da@xarrow{#1}{#2}\da@leftarrow{}{}{\,}}{}%
  }%
}
\newcommand*{\da@xarrow}[7]{%
  \sbox0{$\ifx#7\scriptstyle\scriptscriptstyle\else\scriptstyle\fi#5#1#6\m@th$}%
  \sbox2{$\ifx#7\scriptstyle\scriptscriptstyle\else\scriptstyle\fi#5#2#6\m@th$}%
  \sbox4{$#7\dabar@\m@th$}%
  \dimen@=\wd0 %
  \ifdim\wd2 >\dimen@
    \dimen@=\wd2 %   
  \fi
  \count@=2 %
  \def\da@bars{\dabar@\dabar@}%
  \@whiledim\count@\wd4<\dimen@\do{%
    \advance\count@\@ne
    \expandafter\def\expandafter\da@bars\expandafter{%
      \da@bars
      \dabar@ 
    }%
  }%  
  \mathrel{#3}%
  \mathrel{%   
    \mathop{\da@bars}\limits
    \ifx\\#1\\%
    \else
      _{\copy0}%
    \fi
    \ifx\\#2\\%
    \else
      ^{\copy2}%
    \fi
  }%   
  \mathrel{#4}%
}
\makeatother
\usepackage{geometry}
\title[Griffiths positivity and Top-degree optimal estimates] 
{A characterization of Griffiths semipositivity  \\ by top-degree optimal $\dbar$-estimates}

\author{Shin-ichi MATSUMURA}
\address{
Mathematical Institute 
$\&$ Division for the Establishment of Frontier Science of Organization for Advanced Studies, 
Tohoku University, 
6-3, Aramaki Aza-Aoba, Aoba-ku, Sendai 980-8578, Japan.}
\email{\tt mshinichi-math@tohoku.ac.jp}
\email{\tt mshinichi0@gmail.com}

\date{\today, version 0.01}

\subjclass[2020]{Primary 32L05, Secondary 32W05, 32U05.}

\keywords
{Singular Hermitian metrics, $\dbar$-equations, optimal $L^{2}$-estimates, Griffiths positivity, Nakano positivity.}

\begin{document}

\begin{abstract}
We prove that a singular Hermitian metric $h$ on a vector bundle agrees almost everywhere with a Griffiths semipositive metric if its dual metric $h^*$ is locally integrable and the $\dbar$-equation admits solutions with optimal $L^2$-estimates for $(n,n)$-forms, where $n$ is the dimension of the base manifold. In particular, under this integrability assumption, Nakano semipositivity in the sense of optimal estimates implies Griffiths semipositivity.
\end{abstract}

\maketitle

\section{Introduction}\label{sec-intro}

The theory of singular Hermitian metrics on vector bundles has developed rapidly
(see, for example, \cite{DNWZ23, GMY22, LXYZ26, DNWZ24, IM24, Wat25, Wat26}).
A fundamental difficulty is that curvature cannot, in general, be defined in a suitable sense
(see \cite{Rau15}).
Nevertheless, Griffiths semipositivity can be defined in terms of plurisubharmonicity,
without using curvature.
An analogous analytic formulation of Nakano semipositivity was introduced in \cite{DNWZ23}
(see Definition~\ref{def-top-optimal}).
To recall this formulation, let $E$ be a vector bundle on a complex manifold $X$,
and let $h$ be a singular Hermitian metric on $E$
(see \cite{HPS18} for the precise definition).
Roughly speaking, we say that $h$ is Nakano semipositive if the $\dbar$-equation $\dbar u =g$
for an $E$-valued $(n,q)$-form $g$ can be solved with \textit{optimal $L^2$-estimates}
for every $1 \leq q \leq n$
(see Definition~\ref{def-top-optimal}), where $n:=\dim X$.

When $h$ is smooth, this analytic definition coincides with the usual definition
in terms of the curvature $\sqrt{-1}\Theta_{h}(E)$, as proved in \cite{DNWZ23}.
In this case, Nakano semipositivity implies Griffiths semipositivity directly from the definitions.
However, for singular Hermitian metrics, the corresponding implication has been an open problem
and a major obstacle to developing a theory of Nakano positivity.
This problem is formulated in \cite[Question~7.3]{Ina22}
(see also \cite[Remark~1.9]{DNW21})
and is the main subject of this paper.

Assuming lower semicontinuity of $h$ and strong upper semicontinuity of the dual metric $h^{*}$,
Liu proves in \cite[Theorem~1.7]{Liu26} (see also \cite{Liu25}) that $h$ is Griffiths semipositive
if it satisfies optimal $L^2$-estimates for $(n,1)$-forms or $(1,n)$-forms
(or the corresponding multiple coarse $L^2$-estimate condition).
This gives an affirmative answer to \cite[Question~7.3]{Ina22} under these assumptions.
A key tool is convolution regularization that preserves optimal $L^2$-estimates
for $(n,1)$-forms or $(1,n)$-forms.
His proof uses these semicontinuity assumptions and, in the $(n,1)$-case,
combines optimal $L^2$-estimates for $(n,1)$-forms with induction on the dimension
and the coherence of $\mathcal{E}(h)$.

We prove Griffiths semipositivity from optimal $\dbar$-estimates for $(n,n)$-forms alone,
under the assumption that $h^{*}$ is locally integrable,
without assuming semicontinuity of the metric or optimal $L^2$-estimates for $(n,1)$-forms.

\begin{thm}\label{thm-main}
Let $E$ be a vector bundle on a complex manifold $X$ of dimension $n$.
Let $h$ be a singular Hermitian metric on $E$ whose dual metric $h^*$ is locally integrable
$($i.e., $|u|_{h^{*}}^2\in L^1_{\mathrm{loc}}$ for every local section $u$
of the dual vector bundle $E^{*}$$)$.
If $h$ satisfies the optimal $L^{2}$-estimate condition in Definition~\ref{def-top-optimal} for $q=n$,
then there exists a singular Hermitian metric
$\widehat h$ on $E$ such that $\widehat h$ is Griffiths semipositive
and $\widehat h=h$ almost everywhere.
\end{thm}

In the smooth case, Griffiths semipositivity is equivalent to
the optimal $L^2$-estimate condition for $(n,n)$-forms
with respect to every K\"ahler form
(see \cite{Wat24, Wat25}), so our hypothesis is natural.
Moreover, a Griffiths semipositive singular Hermitian metric can be locally approximated
by smooth metrics using convolution and hence satisfies
optimal $\dbar$-estimates for $(n,n)$-forms.
This result provides further support for the formulation of Nakano semipositivity through optimal estimates.

Our proof uses techniques from \cite{Liu26},
but the main difference is that we cannot use
either induction on the dimension or the coherence of $\mathcal{E}(h)$.
Note that the coherence result for $\mathcal{E}(h)$ requires
optimal $L^2$-estimates for $(n,1)$-forms.

One motivation of this study is to develop a theory of Nakano positivity
for (higher) direct images of adjoint sheaves,
generalizing the theory of Griffiths positivity
established in \cite{Ber09, BP08, PT18, HPS18, TM1, TM2}.
Combined with the optimal $L^2$-estimates of \cite{IMW},
Theorem~\ref{thm-main} recovers, up to almost-everywhere equality,
the Griffiths semipositivity of the canonical $L^2$-metrics on direct images of adjoint bundles
studied in \cite{PT18, HPS18}.

\subsection*{Acknowledgments}\label{subsec-ack}
The author is grateful to Professors Zhuo Liu and Yuta Watanabe for their helpful comments.
He was partially supported by the JST FOREST Program, Grant Number JPMJFR2368.
Generative AI tools were used as a supplementary aid for English-language editing. 
The author takes full responsibility for the content of this paper.

\section{Proof of Theorem \ref{thm-main}}

Before proving Theorem \ref{thm-main},
we recall the definition of Nakano semipositivity in terms of optimal $L^{2}$-estimates.

\begin{defn}\label{def-top-optimal}
Let $X$ be a complex manifold of dimension $n$, and let $E$ be a vector bundle on $X$ of rank $r$.
Let $h$ be a singular Hermitian metric that is finite and positive definite almost everywhere
(see \cite{HPS18} for the precise definition).
The metric $h$ is said to be \textit{Nakano semipositive}
if the following holds: for data consisting of 
\begin{itemize}
\item a Stein coordinate neighbourhood $\Omega\subset X$;
\item a K\"ahler form $\omega_\Omega$ on $\Omega$;
\item a smooth function $\psi$ on $\Omega$ such that $\deldel\psi>0$;
\item a positive integer $q$ with $1\leq q\leq n$;
\item a $\overline{\partial}$-closed $g\in L^2_{n,q}(\Omega,E;\omega_\Omega,he^{-\psi})$
satisfying
$$\int_\Omega\langle B^{-1}_{\omega_\Omega,\psi}g,g\rangle_{\omega_\Omega,h}e^{-\psi}\,dV_{\omega_\Omega}<\infty,$$
\end{itemize}
there exists $u\in L^2_{n,q-1}(\Omega,E;\omega_\Omega,he^{-\psi})$ such that
\[
  \overline{\partial}u=g
  \quad\text{and}\quad
  \int_\Omega |u|^2_{\omega_\Omega,h}e^{-\psi}\,dV_{\omega_\Omega}
  \leq
  \int_\Omega\langle B^{-1}_{\omega_\Omega,\psi}g,g\rangle_{\omega_\Omega,h}
       e^{-\psi}\,dV_{\omega_\Omega},
\]
where $B_{\omega_\Omega,\psi}=[(\deldel\psi)\otimes\Id_E,\Lambda_{\omega_\Omega}]$.
\end{defn}

\begin{remark}\label{main-rem}
The assumption in Theorem \ref{thm-main} consists only of
\textit{top-degree optimal $\dbar$-estimates}.
In particular, the theorem applies to Nakano semipositive metrics in the above sense
whose dual metrics are locally integrable.
The latter condition holds whenever the dual metric is locally essentially bounded.
\end{remark}

We prove Theorem \ref{thm-main}. 

\begin{proof}[Proof of Theorem \ref{thm-main}]
We first outline the proof.
For a smooth Hermitian metric $h$,
the optimal $\dbar$-estimate condition implies Griffiths semipositivity by \cite{DNWZ23}.
To extend this approach to singular metrics, we use the standard approximation of $h$
obtained by taking the dual, convolving, and taking the dual again
(see \eqref{eq-regularization}).
A difficulty is that it is not clear how the optimal $\dbar$-estimate condition behaves under this regularization.
Following the duality and convolution argument of \cite[Section~3]{Liu26},
we focus on the formulation \eqref{eq-dual-inequality}.
In top degree, the curvature operator associated with the weight is scalar,
and this inequality depends linearly on $h^*$.
This linearity allows us to obtain the same inequality for $h^*_{\e}$.
Then, an analogue of the argument of \cite{DNWZ23} gives Griffiths semipositivity of $h_{\e}$.

\smallskip

We now prove Theorem \ref{thm-main}.
The problem is local on $X$.
We may assume that $X$ is an open ball in $\mathbb{C}^{n}$ with coordinates $(z_{1},\dots,z_{n})$
and fix a trivialization of $E$ on $X$.
Let $\omega$ be a K\"ahler form with constant coefficients in the coordinates $(z_{1},\dots,z_{n})$.
Then $\omega$ is invariant under translations $z\mapsto z+c$.
Let $\psi$ be a smooth strictly plurisubharmonic function on $X$.
We write
\[
  dV_\omega:=\frac{\omega^n}{n!},
  \qquad
  B_{\omega,\psi}:=[\idelbar\psi,\Lambda_\omega],
\]
where $\Lambda_\omega$ is the adjoint of exterior product by
$\omega$. 
On $(n,n)$-forms, the curvature operator $B_{\omega,\psi}$ takes the form
\[
  B_{\omega,\psi}=a_\psi\id,
 \]
where $a_\psi$ is the smooth positive function defined by $ a_\psi:=\Lambda_\omega(\idelbar\psi)$.

At almost every point $x\in X$,
we define the antilinear isomorphism
$$
\sharp_h:E_x^*\to E_x
\quad\text{by}\quad
\langle\sigma_x,\sharp_h\sigma_x\rangle=|\sigma_x|_{h^*}^2,
$$
where $h^{*}$ is the dual metric of $h$. 
Let $\sigma\in C_c^\infty(X,E^*)$ be a compactly supported smooth section of $E^*$.
We define the $E$-valued $(n,n)$-form $g$ by
$$
  g:=a_\psi e^\psi\sharp_h(\sigma)\,dV_\omega.
$$
The form $g$ satisfies the assumptions of Definition \ref{def-top-optimal}.
Indeed $g$ is $\dbar$-closed since it has type $(n,n)$, and we have
\[
\begin{split}
  &\int_X |g|_{h,\omega}^2e^{-\psi}\,dV_\omega
  =
  \int_X (a_\psi)^2|\sigma|_{h^*}^2e^\psi\,dV_\omega
  \quad \text{ and }\\
  &L:=\int_X\langle B^{-1}_{\omega,\psi}g,g\rangle_{\omega,h}
       e^{-\psi}\,dV_\omega=
  \int_X
  \frac{|g|_{h,\omega}^2}{a_\psi}e^{-\psi}\,dV_\omega
  =
  \int_X a_\psi|\sigma|_{h^*}^2e^\psi\,dV_\omega.
\end{split}
\]
Both integrals are finite since $h^*$ is locally integrable and $\sigma$ has compact support.
Then, by the optimal $\dbar$-estimate condition, we can find an $E$-valued $(n,n-1)$-form $u$ on $X$
such that
\begin{equation}\label{eq-u-bound}
\dbar u=g
\quad\text{and}\quad
  \int_X |u|_{h,\omega}^2e^{-\psi}\,dV_\omega\leq L.
\end{equation}
Since $h^*$ is locally integrable, the Cauchy--Schwarz inequality implies that
$u$ and $g$ have locally integrable coefficients. 
%Indeed, for every compact subset $K\Subset X$,
%\[
%  \int_K |u|_{0,\omega}\,dV_\omega
%  \leq
%  \left(\int_K |u|_{h,\omega}^2e^{-\psi}\,dV_\omega\right)^{1/2}
%  \left(\int_K \tr(h^*)e^\psi\,dV_\omega\right)^{1/2},
%\]
%where $|\cdot|_{0,\omega}$ uses the Euclidean metric in the chosen trivialization of $E$,
%and $\tr(h^*)$ is the trace in the corresponding dual frame. The same argument applies to $g$.
Thus, we can apply integration by parts in the sense of distributions and the Cauchy--Schwarz inequality. 
Then, we obtain 
\[
\begin{split}
  L
  &=\left|\int_X\langle\sigma,g\rangle\right|
   =\left|\int_X\langle\dbar\sigma,u\rangle\right|\\
  &\leq
  \left(\int_X|\dbar\sigma|_{h^*,\omega}^2e^\psi\,dV_\omega\right)^{1/2}
  \left(\int_X|u|_{h,\omega}^2e^{-\psi}\,dV_\omega\right)^{1/2} \\
  &\leq\left(\int_X|\dbar\sigma|_{h^*,\omega}^2e^\psi\,dV_\omega\right)^{1/2} \cdot L^{1/2}.
\end{split}
\]
This yields 
\begin{equation}\label{eq-dual-inequality}
  \int_X a_\psi|\sigma|_{h^*}^2e^\psi\,dV_\omega
  \leq
  \int_X|\dbar\sigma|_{h^*,\omega}^2e^\psi\,dV_\omega.
\end{equation}
A point is that both sides are linear in the Hermitian form $h^*$.

Choose a nonnegative radial function $\rho\in C_c^\infty(\mathbb{C}^n)$ with
$\int_{\mathbb{C}^n}\rho\,dc=1$, where $dc$ denotes Lebesgue measure, and set
$
  \rho_\e(c):=\e^{-2n}\rho(c/\e).
$
On a smaller coordinate ball $\Omega\Subset X$,
we consider the translation $\tau_c$ defined by $z\mapsto z+c$, 
and define
\begin{equation}\label{eq-regularization}
  h_\e^*:=\int_{\mathbb{C}^n}\rho_\e(c)\tau_c^*h^*\,dc,
  \qquad h_\e:=(h_\e^*)^*.
\end{equation}
These are smooth, positive definite metrics on $\Omega$
for sufficiently small $\e$. Positive definiteness follows since $h^*$ is
positive definite almost everywhere and the kernel is nonnegative.
The regularized metrics are independent of the choice of $\omega$.

We will show that \eqref{eq-dual-inequality} holds locally for the regularized metric $h_{\e}^{*}$ (see \eqref{eq-smooth-dual-inequality}). 
Suppose that $\Omega+c\Subset X$.
For $\sigma\in C_c^\infty(\Omega,E^*)$,
we consider 
\[
  \sigma_c(w)=\sigma(w-c),\qquad \psi_c(w)=\psi(w-c)
\]
on $\Omega+c$. 
Since $\omega$ is translation invariant,
\[
  \tau_c^*dV_\omega=dV_\omega,\qquad
  \tau_c^*a_{\psi_c}=a_\psi,\qquad
  \tau_c^*(\dbar\sigma_c)=\dbar\sigma.
\]
By applying \eqref{eq-dual-inequality} on $\Omega+c$ to $h^*$, $\sigma_c$, and
$\psi_c$ and substituting $w=z+c$, we obtain 
\begin{equation}\label{eq-translated}
  \int_\Omega a_\psi|\sigma|_{\tau_c^*h^*}^2e^\psi\,dV_\omega
  \leq\int_\Omega|\dbar\sigma|_{\tau_c^*h^*,\omega}^2e^\psi\,dV_\omega.
\end{equation}
This means that the estimate applies on the translated domain with the translated
weight. For each fixed vector or form, the squared norm is linear in
the Hermitian metric. Consequently, we see that
\[
\begin{split}
  |\sigma|_{h_\e^*}^2
  &=\int_{\mathbb{C}^n}\rho_\e(c)|\sigma|_{\tau_c^*h^*}^2\,dc,\\
  |\dbar\sigma|_{h_\e^*,\omega}^2
  &=\int_{\mathbb{C}^n}\rho_\e(c)|\dbar\sigma|_{\tau_c^*h^*,\omega}^2\,dc
\end{split}
\]
hold pointwise on $\Omega$. 
Multiplying \eqref{eq-translated} by $\rho_\e(c)$ and integrating
in $c$, we obtain 
\begin{equation}\label{eq-smooth-dual-inequality}
  \int_\Omega a_\psi|\sigma|_{h_\e^*}^2e^\psi\,dV_\omega
  \leq
  \int_\Omega|\dbar\sigma|_{h_\e^*,\omega}^2e^\psi\,dV_\omega
\end{equation}
by Fubini's theorem.

Let $D'_{h_\e^*e^\psi}$ be the $(1,0)$-part of the Chern connection of $h_\e^*e^\psi$. 
The Bochner formula shows that 
\[
  \|\dbar\sigma\|_{h_\e^*e^\psi,\omega}^2
  =\|D'_{h_\e^*e^\psi}\sigma\|_{h_\e^*e^\psi,\omega}^2
  -\int_\Omega
  \la\Lambda_\omega(\sqrt{-1}\Theta_{h_\e^*e^\psi})\sigma,\sigma\ra_{h_\e^*e^\psi}\,dV_\omega.
\]
By 
$
  \sqrt{-1}\Theta_{h_\e^*e^\psi}
  =\sqrt{-1}\Theta_{h_\e^*}-\idelbar\psi\,\id,
$
and \eqref{eq-smooth-dual-inequality}, 
we obtain 
\begin{equation}\label{eq-bochner-bound}
  0\leq\|D'_{h_\e^*e^\psi}\sigma\|_{h_\e^*e^\psi,\omega}^2
  -\int_\Omega
  \la\Lambda_\omega(\sqrt{-1}\Theta_{h_\e^*})\sigma,\sigma\ra_{h_\e^*e^\psi}\,dV_\omega.
\end{equation}

We will deduce that $h_{\e}$ is Griffiths semipositive from \eqref{eq-bochner-bound}. 
Fix a point $x\in\Omega$ and a vector $v_x\in E_x$. 
We take a holomorphic section $v$ on $\Omega$ with $v(x)=v_x$. 
Let $\chi\in C_c^\infty(\Omega,\R)$ be equal to $1$ near $x$. 
We define a compactly supported smooth section $\sigma$ of $E^{*}$ by
\[
  \sigma_y(w):=\chi(y)e^{-\psi(y)}h_{\e,y}(w,v(y))
  \quad (y\in\Omega,\ w\in E_y).
\]
Since $v$ is holomorphic, we have 
$D'_{h_\e^*}h_\e(\bullet,v)=0$. Hence, we obtain 
\[
  D'_{h_\e^*e^\psi}\sigma=e^{-\psi}\partial\chi\cdot h_\e(\bullet,v).
\]
Using the isomorphism $\sharp_{h_\e}$ and the curvature identity for the dual metric, we obtain
\begin{equation}\label{eq-cutoff-curvature}
  0\leq\int_\Omega
  \left(
    |\partial\chi|_\omega^2|v|_{h_\e}^2
    +\chi^2\la\Lambda_\omega(\sqrt{-1}\Theta_{h_\e})v,v\ra_{h_\e}
  \right)e^{-\psi}\,dV_\omega.
\end{equation}

Consider the strictly psh function $\psi_t(y):=t\|y-x\|^2$ on $\mathbb{C}^n$,
where $\|\cdot\|$ denotes the standard Euclidean norm, and put
\[
  Z_t:=\int_\Omega e^{-\psi_t}\,dV_\omega\asymp t^{-n} \quad (t \to \infty).
\]
Here, the asymptotic behavior follows from the change of variables $w=\sqrt{t}(y-x)$.
From now on, we fix $\e$ and $\omega$. Set
\[
  F=\la\Lambda_\omega(\sqrt{-1}\Theta_{h_\e})v,v\ra_{h_\e}.
\]
We can choose $\delta>0$ such that $\|y-x\|\geq\delta$ on $\Supp(\partial\chi)$. 
Therefore, after division by $Z_t$, the first term of \eqref{eq-cutoff-curvature} converges to zero:
\[
  0\leq\frac{1}{Z_t}\int_\Omega |\partial\chi|_\omega^2|v|_{h_\e}^2 e^{-\psi_t}\,dV_\omega
  \leq Ct^n e^{-t\delta^2}\longrightarrow 0.
\]
For the second term involving the curvature, we consider the change of variables $w=\sqrt{t}(y-x)$. 
Then, we obtain 
\[
\begin{split}
  t^n Z_t
  &=\int_{\sqrt{t}(\Omega-x)}e^{-\|w\|^2}\,dV_\omega(w),\\
  t^n\int_\Omega\chi^2Fe^{-\psi_t}\,dV_\omega
  &=\int_{\sqrt{t}(\Omega-x)}
  (\chi^2F)(x+w/\sqrt{t})e^{-\|w\|^2}\,dV_\omega(w).
\end{split}
\]
Extend $\chi^2F$ by zero outside $\Omega$.
Then $(\chi^2F)(x+w/\sqrt{t})$ is uniformly bounded and converges pointwise to $F(x)$,
since $\chi=1$ near $x$.
Since $\omega$ has constant coefficients, $dV_\omega$ is a constant multiple of Lebesgue measure.
The dominated convergence theorem therefore shows that
the first integral converges to $C_\omega=\int_{\mathbb{C}^n}e^{-\|w\|^2}\,dV_\omega>0$ 
and the second integral converges to $C_\omega F(x)$. 
Hence, we obtain 
\[
  \frac{1}{Z_t}\int_\Omega\chi^2Fe^{-\psi_t}\,dV_\omega\longrightarrow F(x).
\]
Dividing \eqref{eq-cutoff-curvature} by $Z_t$ and taking the limit, we can conclude that 
\begin{equation}\label{eq-trace-positive}
F(x)=
  \la\Lambda_\omega(\sqrt{-1}\Theta_{h_\e})(x)v_x,v_x\ra_{h_\e}\geq 0.
\end{equation}

Consequently, for fixed $x$ and $v_x$, the scalar $(1,1)$-form
\[
  Q:=\la\sqrt{-1}\Theta_{h_\e}(x)v_x,v_x\ra_{h_\e}
\]
satisfies $\Lambda_\omega Q\geq 0$ for 
any K\"ahler form $\omega$ with constant coefficients. 
Thus $Q\geq0$, and $h_\e$ is Griffiths semipositive.
Equivalently, $\log|u|_{h_\e^*}$ is psh for every local section $u$ of $E^*$.
For every section $a$ of $E^*$, the functions $|a|_{h_\e^*}^2$ are psh
and converge in $L^1_{\mathrm{loc}}$ to $f_a:=|a|_{h^*}^2$.
Hence $f_a$ admits a psh representative on $X$, since $\Omega\Subset X$ was arbitrary.
Since $\rho$ is radial, its convolutions
$f_a*\rho_\e=|a|_{h_\e^*}^2$ decrease as $\e\downarrow0$.
Consequently, we see that $h_\e^*$ decreases to a singular Hermitian metric
$\widehat h^*$ with $\widehat h^*=h^*$ almost everywhere.
For every local holomorphic section $u$ of $E^*$, the psh functions
$\log|u|_{h_\e^*}$ decrease to $\log|u|_{\widehat h^*}$, which is  psh.
Thus $\widehat h^*$ is Griffiths seminegative, and its dual $\widehat h$ is Griffiths semipositive
and satisfies $\widehat h=h$ almost everywhere.
\end{proof}

\bibliographystyle{alpha}

\begin{thebibliography}{99}
\bibitem[Ber09]{Ber09}
B.~Berndtsson,
\textit{Curvature of vector bundles associated to holomorphic fibrations},
Ann. of Math. (2) \textbf{169} (2009), no.~2, 531--560.

\bibitem[BP08]{BP08}
B.~Berndtsson and M.~P{\u{a}}un,
\textit{Bergman kernels and the pseudoeffectivity of relative canonical bundles},
Duke Math. J. \textbf{145} (2008), no.~2, 341--378.

\bibitem[DNW21]{DNW21}
F.~Deng, J.~Ning and Z.~Wang,
\textit{Characterizations of plurisubharmonic functions},
Sci. China Math. \textbf{64} (2021), no.~9, 1959--1970.

\bibitem[DNWZ23]{DNWZ23}
F.~Deng, J.~Ning, Z.~Wang and X.~Zhou,
\textit{Positivity of holomorphic vector bundles in terms of
$L^p$-estimates for $\overline{\partial}$},
Math. Ann. \textbf{385} (2023), no.~1--2, 575--607.

\bibitem[DNWZ24]{DNWZ24}
F.~Deng, J.~Ning, Z.~Wang and X.~Zhou,
\textit{Approximation and extension of Hermitian metrics on holomorphic
vector bundles over Stein manifolds},
C. R. Math. Acad. Sci. Paris \textbf{362} (2024), 1707--1715.

\bibitem[GMY22]{GMY22}
Q.~Guan, Z.~Mi and Z.~Yuan,
\textit{Optimal $L^2$ extension for holomorphic vector bundles
with singular Hermitian metrics},
Peking Math. J. \textbf{9} (2026), no.~1, 105--193.

\bibitem[HPS18]{HPS18}
C.~Hacon, M.~Popa and C.~Schnell,
\textit{Algebraic fiber spaces over abelian varieties:
around a recent theorem by Cao and P{\u{a}}un},
in Local and global methods in algebraic geometry,
Contemp. Math., vol.~712,
Amer. Math. Soc., Providence, RI, 2018, pp.~143--195.

\bibitem[Ina22]{Ina22}
T.~Inayama,
\textit{Nakano positivity of singular Hermitian metrics and vanishing
theorems of Demailly--Nadel--Nakano type},
Algebr. Geom. \textbf{9} (2022), no.~1, 69--92.

\bibitem[IM24]{IM24}
T.~Inayama and S.~Matsumura,
\textit{Nakano positivity of singular Hermitian metrics:
Approximations and applications},
J. Funct. Anal. \textbf{290} (2026), no.~1, 111206.

\bibitem[IMW24]{IMW}
T.~Inayama, S.~Matsumura and Y.~Watanabe,
\textit{Singular Nakano positivity of direct image sheaves of adjoint bundles},
to appear in Anal. PDE; arXiv:2407.11412 (2024).

\bibitem[Liu-a]{Liu25}
Z.~Liu,
\emph{An optimal $L^{2}$ extension for continuous $L^{2}$-optimal Hermitian metrics},
preprint, arXiv:2506.22840v1 (2025).

\bibitem[Liu-b]{Liu26}
Z.~Liu,
\emph{Convolution regularization preserves the $L^2$-estimate property
for $(1,n)$-forms},
preprint, arXiv:2506.22834v2 (2026).

\bibitem[LXYZ26]{LXYZ26}
Z.~Liu, B.~Xiao, H.~Yang and X.~Zhou,
\textit{Multiplier submodule sheaves, $L^2$ extension and a problem of Lempert},
J. Reine Angew. Math. \textbf{833} (2026), 111--157.

\bibitem[MT08]{TM1}
C.~Mourougane and S.~Takayama,
\textit{Hodge metrics and the curvature of higher direct images},
Ann. Sci. \'Ec. Norm. Sup\'er. (4) \textbf{41} (2008), no.~6, 905--924.

\bibitem[MT09]{TM2}
C.~Mourougane and S.~Takayama,
\textit{Extension of twisted Hodge metrics for K\"ahler morphisms},
J. Differential Geom. \textbf{83} (2009), no.~1, 131--161.

\bibitem[PT18]{PT18}
M.~P{\u{a}}un and S.~Takayama,
\textit{Positivity of twisted relative pluricanonical bundles
and their direct images},
J. Algebraic Geom. \textbf{27} (2018), no.~2, 211--272.

\bibitem[Rau15]{Rau15}
H.~Raufi,
\textit{Singular Hermitian metrics on holomorphic vector bundles},
Ark. Mat. \textbf{53} (2015), no.~2, 359--382.

\bibitem[Wat24]{Wat24}
Y.~Watanabe,
\emph{$\omega$-trace and Griffiths positivity for singular Hermitian metrics},
preprint, arXiv:2402.06658 (2024).

\bibitem[Wat25]{Wat25}
Y.~Watanabe,
\emph{Curvature operator of holomorphic vector bundles and
$L^2$-estimate condition for $(n,q)$ and $(p,n)$-forms},
Tohoku Math. J. (2) \textbf{77} (2025), no.~4, 445--471.

\bibitem[Wat26]{Wat26}
Y.~Watanabe,
\emph{Griffiths and Nakano positivity and Nakano-Nadel vanishing theorems for singular Hermitian metrics on complex spaces},
preprint, arXiv:2606.16275 (2024).



\end{thebibliography}

\end{document}